\documentclass{my-aims}

\usepackage{txfonts}
\def\typeofarticle{Research article}

\def\DOI{10.3934/math.2022xxx}
\def\Received{4 May 2022}
\def\Revised{18 June 2022}
\def\Accepted{5 July 2022}

\numberwithin{equation}{section}

\makeatletter
\renewcommand{\@biblabel}[1]{#1\hfill \hspace{-0.2cm}}
\makeatother

\usepackage{amsmath,amssymb,amsthm}
\usepackage{url}
\usepackage{adjustbox}

\usepackage{listings}

\newtheorem{theorem}{Theorem}
\newtheorem{lemma}[theorem]{Lemma}

\begin{document}

\title{}

\title{Mathematical Analysis, Forecasting and Optimal Control 
of HIV/AIDS Spatiotemporal Transmission with a Reaction Diffusion SICA Model}


\author{Houssine Zine\affil{1}, 
Abderrahim El Adraoui\affil{2}  
and Delfim F. M. Torres\affil{1,}\corrauth} 

\shortauthors{the Author(s)}

\address{\addr{\affilnum{1}}{Center for Research and Development in Mathematics and Applications (CIDMA),\\
Department of Mathematics, University of Aveiro, 3810-193 Aveiro, Portugal}
\addr{\affilnum{2}}{Laboratory of Analysis Modeling and Simulation (LAMS), 
Department of Mathematics and Computer Science, Faculty of Sciences Ben M'Sik, 
Hassan II University of Casablanca, Morocco}}

\corraddr{delfim@ua.pt}


\begin{abstract}
We propose a mathematical spatiotemporal epidemic SICA model 
with a control strategy. The spatial behavior is modeled 
by adding a diffusion term with the Laplace operator, 
which is justified and interpreted both mathematically and physically.
By applying semigroup theory on the ordinary differential equations, 
we prove existence and uniqueness of the global positive spatiotemporal 
solution for our proposed system and some of its important characteristics. 
Some illustrative numerical simulations are carried out 
that motivate us to consider optimal control theory. A suitable optimal control 
problem is then posed and investigated. Using an effective method based 
on some properties within the weak topology, we prove existence 
of an optimal control and develop an appropriate set 
of necessary optimality conditions to find the optimal control pair that 
minimizes the density of infected individuals and the cost 
of the treatment program.
\end{abstract}

\keywords{HIV/AIDS epidemiology; reaction-diffusion; spatiotemporal SICA model;
optimal control strategies; necessary optimality conditions.\\[0.10cm]
}

\maketitle


\section{Introduction}

The human immunodeficiency virus (HIV) causes millions of deaths to humans worldwide, 
being one of the most infectious and deadly virus \cite{MR4376326}.
The deterministic SICA model was introduced by Silva and Torres
in 2015, as a sub-model of a general Tuberculosis and HIV/AIDS 
(acquired immunodeficiency syndrome) co-infection problem \cite{9}. 
After that, it has been extensively used to investigate HIV/AIDS,
in different settings and contexts, using fractional-order derivatives \cite{12}, 
stochasticity \cite{13} and discrete-time operators \cite{MR4276109}, 
and adjusted to different HIV/AIDS epidemics, as those in Cape Verde \cite{10} 
and Morocco \cite{MR3999700}. 

One of the fundamental characteristics of SICA modeling is that it provides 
adequate but simple mathematical models that help to characterize and understand some 
of the essential epidemiological factors leading to the spreed of the AIDS disease.
In such models, the susceptible population $S$ is nourished by the recruitment of individuals 
into the population at a rate $\lambda$. All individuals are exposed to natural death, 
at a constant rate $\mu$. Individuals $S$ are susceptible to HIV infection 
from an effective contact with an individual carrying the HIV, at the rate 
$\dfrac{\beta}{N}\left(I+\eta_C C+\eta_A A\right)$,
where $I$, $C$ and $A$ denote, respectively, the infected, chronic (under treatment) 
and AIDS individuals, $N$ represents the total number of individuals in the population under study,
that is, $N$ is the sum of $S$, $I$, $C$ and $A$ individuals, and 
$\beta$, $\eta_C$ and $\eta_A$ are parameters that depend on the particular
situation under study. For a survey on SICA models for HIV transmission,
showing that they provide a good framework for interventions and strategies to fight against 
the transmission of the HIV/AIDS epidemic, we refer the reader to \cite{MR4175342}.

It is well known that reaction-diffusion equations are commonly used to model a variety 
of physical and biological phenomena \cite{Laaroussi2020,15,16,Smoller,Wang2,Wang1}. 
Such equations describe how the concentration or density distributed in space varies 
under the influence of two processes: (i) local interactions of species and 
(ii) diffusion, which causes the spread of species in space. 
Recently, reaction-diffusion equations have been used by many authors in epidemiology 
as well as virology, see, e.g., \cite{14}, where a mathematical model is proposed 
to simulate the hepatitis B virus infection with spatial dependence,
or the non-theoretical reviews \cite{Ewald,Hufsky}:
in \cite{Ewald}, host-pathogen interactions are described by 
different temporal and spatial scales, while \cite{Hufsky}
covers bioinformatics workflows and tools for the routine detection 
of the SARS-CoV-2 infection. Here we propose, for the first time in the literature,
to use SICA modeling with $S$, $I$, $C$ and $A$ (thus, also $N$) as functions of both time $t$ 
and space $x$. The spatial effect plays a crucial role in the spread of the virus. 
In order to well describe this phenomenon, we incorporate terms that model
the spatial diffusion in each compartment, by adding $\Delta S$,
$\Delta I$, $\Delta C$ and $\Delta A$ in the classical SICA model system. 
By taking into account the spatiotemporal diffusion allow us not to neglect 
a good part of compartments' inputs-outputs.

The paper is organized as follows. We begin with some preliminaries 
on the physical interpretation of the Laplacian in Section~\ref{sec:2}. 
The spatiotemporal SICA model is then introduced in Section~\ref{sec:3}
and its mathematical analysis is given in Section~\ref{sec:4}
where, by using semigroup theory \cite{17,18}, we prove existence 
and uniqueness of a strong nonnegative solution to the system
(see Theorem~\ref{thm}). In Section~\ref{sec:5}, we show some numerical 
examples that motivate us to consider optimal control.
An optimal control problem is then formulated and existence of a solution is 
established (see Theorem~\ref{thm:02}). Next, we obtain in Section~\ref{sec:6} 
a set of necessary optimality conditions that characterize the optimal solution.
We end with Section~\ref{sec:7} of conclusions, pointing also some future
directions of research.


\section{Preliminaries: interpretation of the Laplacian}
\label{sec:2}

Let $\nabla^2$ be the Laplacian in two dimensions expressed by
$$
\nabla^2 = \frac{\partial^2}{\partial x^2} + \frac{\partial^2}{\partial y^2}.
$$
Suppose that, at a point $O$, taken as the origin of the system of axises $Oxy$, 
a field $f$ takes the value $f_0$. Consider an elementary square with side $a$ 
whose edges are parallel to the coordinate axises and whose center merges 
with the origin $O$. The average value of $f$ in this elementary cube, that is, 
the mean value of $f$ in the neighborhood of the point $O$, is given by the expression 
$$
\overline{f} = \frac1{a^2} \int_\mathcal C f(x,y)\;\mathrm dx \mathrm dy,
$$
where the two integrations relate to the rectangle 
$C = [-\frac{a}{2}, \frac {a}{2}]^2 $. At an arbitrary point $P(x, y)$ 
in the neighborhood of $O = (0,0)$, we develop $f$ in Taylor--Maclaurin series. 
Thus,
\begin{align*}
f(x,y) = f_0 
+ \left(\frac{\partial f}{\partial x}\right)_0 x 
+ \left(\frac{\partial f}{\partial y}\right)_0 y 
+ \frac{1}{2} \left[
\left(\frac{\partial^2f}{\partial x^2}\right)_0 x^2 
+ \left(\frac{\partial^2f}{\partial y^2}\right)_0 y^2  \right] 
+ \left(\frac{\partial^2f}{\partial x\partial y}\right)_0 xy 
+ O(x^2+y^2).
\end{align*}
On one hand, the odd functions in this expression provide, by integration from 
$-\frac{a}{2}$ to $\frac {a}{2}$, a zero contribution to $\overline{f}$. For example,
$$
\int_\mathcal C x\;\mathrm dx \mathrm dy
= \left( \frac{\left(\frac a2\right)^2}2 
- \frac{\left(\frac{-a}2\right)^2}2\right) 
\left(\frac a2 - \frac{-a}2\right) = 0.
$$
On the other hand, each even function provide 
a contribution of $\frac{a^4}{12}$. For example,
$$
\int_\mathcal C x^2\;\mathrm dx \mathrm dy  
= \left(\frac{\left(\frac a2\right)^3}3 
- \frac{\left(\frac{-a}2\right)^3}3\right) 
\left(\frac a2 - \frac{-a}2\right)  = \frac{a^4}{12}.
$$ 
Using the Fubini--Tonnelli theorem, we get
$$
\int_\mathcal C xy\;\mathrm dx \mathrm dy=0.
$$
We deduce that
$$
\overline f \approx  f_0 + \frac{a^4}{24}
\left(\frac{\partial^2f}{\partial x^2} 
+ \frac{\partial^2f}{\partial y^2}\right)_0
$$
and 
$$
\overline f \approx  f_0 + \frac{a^4}{24} \bigl(\nabla^2f\bigr)_0.
$$
As the point $O$ has been chosen arbitrarily, we can assimilate 
it to the current point $P$ and drop the index $0$. Therefore, we 
obtain the expression
$$
\nabla^2f\approx \frac{24}{a^4} \left(\overline f - f\right),
$$
the interpretation of which is immediate: the quantity $\nabla^2 f$ 
is approximately proportional to the difference $\overline {f} - f$. 
The constant of proportionality is worth $\frac{24}{a^4}$ 
in Cartesian axises. In other words, the quantity $\nabla^2 f$ 
is a measure of the difference between the value of $f$ 
at any point $P$ and the mean value $\overline{f}$ 
in the neighborhood of point $P$. 


\section{The spatiotemporal mathematical SICA model}
\label{sec:3}

In \cite{10}, Silva and Torres proposed the following epidemic SICA model:
\begin{equation}
\label{SICA:ST:Ref10}
\left\lbrace
\begin{aligned}
\frac{dS(t)}{dt} 
&=  \Lambda -\beta \left( I (t) +\eta _{C}\cdot C(t) 
+\eta_{A} \cdot A (t) \right) \cdot S (t) -\mu S (t),\\ 
\frac{dI(t)}{dt} 
&= \beta \left( I (t) +\eta_{C}
\cdot C\left(t\right) +\eta _{A} \cdot A (t) \right) \cdot S (t) 
-\xi_{3} I (t) +\gamma A (t) +\omega C (t),\\ 
\frac{dC(t)}{dt}
&= \phi I (t) -\xi_{2}C\left( t\right),\\ 
\frac{dA(t)}{dt} 
&= \rho I (t) -\xi_{1}A\left( t\right).
\end{aligned}
\right.
\end{equation}
The limitation of the temporal dynamical system \eqref{SICA:ST:Ref10} 
to give a good description of the spread of the virus in the space is obvious. 
To bridge this gap, we suggest to use of the Laplacian operator 
as interpreted in Section~\ref{sec:2}. In concrete,
we extend the deterministic epidemic SICA model 
\eqref{SICA:ST:Ref10} as follows:
\begin{equation}
\label{SICA2}
\left\lbrace
\begin{aligned}
\frac{\partial S (t,x)}{\partial t} 
&= d_S\Delta S (t,x)+ \Lambda -\beta \left( I (t,x) 
+\eta _{C}\cdot C(t,x) +\eta _{A}\cdot A(t,x) \right)
\cdot S (t,x) -\mu S(t,x)\\
&\qquad +u (t,x) I(t,x),\\ 
\frac{\partial I (t,x)}{\partial t} 
&=d_I\Delta I (t,x) + \beta \left( I (t,x) +\eta _{C}\cdot C(t,x) 
+\eta_{A}\cdot A(t,x) \right)\cdot S (t,x) -\xi_{3}I (t,x) 
+\gamma A (t,x) \\
&\qquad +\omega C (t,x) -u (t,x)I (t,x),\\ 
\frac{\partial C (t,x)}{\partial t}
&=d_C\Delta C (t,x) + \phi I (t,x) -\xi_{2}C( t,x),\\ 
\frac{\partial A (t,x)}{\partial t} 
&=d_A\Delta A (t,x) + \rho I (t,x) -\xi_{1}A(t,x),
\end{aligned}
\right.
\end{equation}
where $\Delta$ is the Laplacian 
in the two-dimensional space $(t,x)$ 
and $u:[0;T]\times \Omega\longrightarrow [0;1[$ 
is a control that permits to diminish the number 
of infected individuals and to increase that of susceptible by devoting 
some special treatment to the most affected persons. The description 
of the parameters of model \eqref{SICA2} is summarized in Table~\ref{tab:1}.
\begin{table}[ht!]
\caption{Description of the parameters of the 
spatiotemporal SICA epidemic model \eqref{SICA2}.}
\label{tab:1}
\begin{center}
\begin{tabular}{lll}
Symbol&&Description\\ \hline
$\Lambda$&&Recruitment rate\\
$\mu$&&Natural death rate\\
$\beta$&&HIV transmission rate\\
$\eta_C$&&Modification parameter\\
$\eta_A$&&Modification parameter\\
$\phi$&&HIV treatment rate for $I$ individuals\\
$\rho$&&Default treatment rate for $I$ individuals\\
$\gamma$&&AIDS treatment rate\\
$\omega$&&Default treatment rate for $C$ individuals\\
$d$&&AIDS induced death rate\\
$d_S$&&Diffusion of susceptible individuals\\
$d_I$&&Diffusion of infected individuals with no AIDS symptoms\\
$d_C$&&Diffusion of chronic individuals\\
$d_A$&&Diffusion of infected individuals with AIDS symptoms\\
\hline
\end{tabular}
\end{center}
\end{table}

\section{Existence and uniqueness of a strong nonnegative solution}
\label{sec:4}

In order to prove existence and uniqueness of a strong solution to system \eqref{SICA2},
we define some tools. Consider the Hilbert spaces $H(\Omega)= (L_2(\Omega))^4$,  
$H^1(\Omega)=\left\{u\in L_2(\Omega):\; \dfrac{\partial u}{\partial x}\in L_2(\Omega)\;
\text{and}\;\dfrac{\partial u}{\partial y}\in L_2(\Omega)\right\}$ 
and $H^2(\Omega)=\left\{u\in H^1(\Omega):\; \dfrac{\partial^2 u}{\partial x^2},
\dfrac{\partial^2 u}{\partial y^2}, \dfrac{\partial^2 u}{\partial x\partial y}, 
\dfrac{\partial^2 u}{\partial y\partial x} \in L_2(\Omega)\right\}$.
Let $L^2(0,T;H^2(\Omega))$ be the space of all strongly 
measurable functions $v:[0,T]\longmapsto H^2(\Omega)$ such that
$$
\int\limits_{0}^{T} \lVert v (t,x)\rVert _{H^2(\Omega)}\;\mathrm dt<\infty
$$
and $L^{\infty}( 0, T ; H^1(\Omega))$ be the set of all functions 
$v:[0,T]\longmapsto H^1(\Omega)$ verifying
$$
\underset{t\in [0,T]}{\sup}(\lVert v(t,x)\rVert _{H^1(\Omega)})<\infty.
$$
The norm in $L^{\infty}( 0, T ; H^1(\Omega))$ is defined by
$$
\lVert v \rVert_{L^{\infty}( 0, T ; H^1(\Omega))} 
:=\inf\left\{c\in \mathbb{R}_+
: \lVert v (t,x)\rVert _{H^1(\Omega)}<c\right\}.
$$

Our model is equivalent to 
\begin{equation}
\label{eqz}
\frac{\partial z (t,x)}{\partial t}=Az (t,x)+g(t,z (t,x)),
\end{equation}
where $z=(z_1,z_2,z_3,z_4)=(S,I,C,A)$ and $g=(g_1,g_2,g_3,g_4)$ is defined by
\begin{equation*}
\begin{cases}
g_1=-\beta(z_2+\eta_Cz_3+\eta_Az_1)z_1-\mu z_1+\Lambda+uz_2,\\
g_2=\beta(z_2+\eta_Cz_3+\eta_Az_1)z_1-\xi_3 z_2+\gamma z_4+\omega z_3 -uz_2,\\
g_3=\varPhi  z_2-\xi_2 z_3,\\
g_4=\rho z_2-\xi_1z_4.
\end{cases}
\end{equation*}
For all $i\in \{1,2,3,4\}$, 
$$
\frac{\partial z_i}{\partial t}=d_i\Delta z_i+g_i(z (t,x)).
$$
Let $A$ denote the linear operator defined from 
$D(A)\subset H(\Omega)$ to $H(\Omega)$ by 
$$
Az=\left(d_S \triangle z_1,d_I \triangle z_2,
d_C \triangle z_3,d_A \triangle z_4\right)
$$
with
$$
z \in D(A)=
\left\{z=(z_1,z_2,z_3,z_4)\in \left( H^2(\Omega)\right)^4 
:\;\dfrac{\partial z_1}{\partial \eta}=\dfrac{\partial z_2}{\partial \eta}
=\dfrac{\partial z_3}{\partial \eta}=\dfrac{\partial z_4}{\partial \eta}=0 
\quad \text{on} \quad  \partial \Omega \right\}
$$
and $U_{ad}$ be the admissible control set defined by
\begin{equation}
\label{Uad}
U_{ad}=\left\{u\in L^2(Q),0 \leq u \leq 1 \quad \text{a.e. on }Q\right\}
\end{equation}
with $Q=[0,T]\times \Omega$ and $\Omega$  a bounded domain in $\mathbb{R}^2$ 
with smooth boundary $\partial \Omega$.

To obtain our next result, we employ semi-group theory \cite{18}
to prove existence and uniqueness of a global 
nonnegative solution to the considered system.

\begin{theorem}
\label{thm}
Let $\Omega$ be a bounded domain from $\mathbb{R}^2$ 
with a boundary of class $C^{2+\alpha}$, $\alpha>0$. 
For nonnegative parameters of the spatiotemporal SICA model \eqref{SICA2}, 
$u\in U_{ad}$, $z^0\in D(A)$ and $z^0_i\geq 0$ on $\Omega$, $i=1,2,3,4$, 
the system \eqref{SICA2} has a unique (global) strong nonnegative 
solution $z \in W^{1,2}([0,T];H(\Omega))$ such that 
$$ 
z_1,z_2,z_3,z_4 \in L^2( 0, T ; H^2 (\Omega)) 
\cap L^{\infty} (0, T ; H^1 (\Omega))\cap L^\infty(Q).
$$
Additionally, there exists $C>0$, independent of $u$ and of the corresponding solution $z$, 
such that for all $t \in [0,T]$ and all $i\in \{1,2,3,4\}$ one has
$$
\left\| \frac{ \partial z_i}{\partial t}\right\|_{L^2(Q)}
+ \left\| z_i \right\|_{L^2(0,T,H^2(\Omega))}
+ \left\|z_i \right\|_{H^1(\Omega)}
+ \left\|z_i\right\|_{H^\infty(Q)} \leq C.
$$
\end{theorem}

\begin{proof}
Because the Laplacian operator $\Delta$ is dissipating, self-adjoint,
and generates a $C_0-$ semigroup of contractions on $ H(\Omega)$, 
it is clear that function $g = (g_1 , g_2 , g_3,g_4 )$ becomes Lipschitz continuous 
in $z = (z_1 , z_2 , z_3,z_4 )$ uniformly with respect
to $t \in [0, T]$. Therefore, the problem admits a unique strong solution $z$.
Let us now show that for all $i\in \{1,2,3,4\}$,\;\;$z_i\in L_{\infty}(Q)$.
Indeed, set $k=\max\left\{\|g_i\|_{L_{\infty}(Q}),\|z_i^0\|_{L_{\infty}(\Omega)}\;
:\;i\in \{1,2,3,4\} \right\}$ and let 
$$
U_i (t,x)=z_i (t,x)-kt-\|z_i^0\|_{L^{\infty}(\Omega)}.
$$
Then, 
\begin{equation*}
\begin{cases}
\dfrac{\partial U_i (t,x)}{\partial t}
=d_i\Delta U_i (t,x)+g_i(t,z (t,x))-k,\;\;\;t\in[0,T],\\
U_i(0,x,y)=z_i^0-\|z_i^0\|_{L^{\infty}(\Omega)}.
\end{cases}
\end{equation*}
Let $i\in \{1,2,3,4\}$. There exists an infinitesimal semigroup $\Gamma(t)$ 
associated to the operator $d_i\Delta$ such that
$$
U_i(t,x)=\Gamma(t)\left(z_i^0-\|z_i^0\|_{L^{\infty}(\Omega)}\right)
+\int_{0}^{t}\Gamma(t-s)\left(g_i(z(s))-k\right)ds.
$$
We deduce that $U_i (t,x)\leq 0$ and so $z_i\leq kt+\|z_i^0\|_{L^{\infty}(\Omega)}$.

Consider $V_i (t,x)=z_i (t,x)+kt+\|z_i^0\|_{L^{\infty}(\Omega)}$. Upon differentiation, 
we get
\begin{equation*}
\begin{cases}
\dfrac{\partial V_i (t,x)}{\partial t}
=d_i\Delta V_i (t,x)+g_i(t,z (t,x))+k,\;\;\;t\in[0,T],\\
V_i(0,x,y)=z_i^0+\|z_i^0\|_{L^{\infty}(\Omega)}.
\end{cases}
\end{equation*}
The strong solution of the above equation is
$$
V_i (t,x)=\Gamma(t)\left(z_i^0+\|z_i^0\|_{L^{\infty}(\Omega)}\right)
+\int_{0}^{t}\Gamma(t-s)\left(g_i(z(s))+k\right)ds.
$$
Then, $V_i (t,x)\geq 0$ and so $z_i\geq -kt-\|z_i^0\|_{L^{\infty}(\Omega)}$. 
Consequently, $|z_i(t,x,)|\leq kt+\|z_i^0\|_{L^{\infty}(\Omega)}$, 
which implies that $z_i\in L_{\infty}(Q)$.

Now, we proceed by proving that $z_i\in L_{\infty}\left(0,T;H^1(\Omega)\right)$
for all $i\in \{1,2,3,4\}$. Indeed, let $i\in \{1,2,3,4\}$. From equality
$$
\dfrac{\partial z_i (t,x)}{\partial t}-d_i\Delta z_i (t,x)
=g_i(t,z (t,x)),\;\;\;(t,x)\in[0,T]\times \Omega,
$$
we obtain that
$$
\int_{0}^{t}\int_{\Omega}\left(\dfrac{\partial z_i (t,x)}{\partial t}
-d_i\Delta z_i (t,x)\right)^2dx ds
=\int_{0}^{t}\int_{\Omega}\left(g_i(t,z (t,x))\right)^2dx ds.
$$
From Green's formula, we get
\begin{align*}
\int_{0}^{t}\int_{\Omega}\left(\dfrac{\partial z_i}{\partial t}\right)^2dx ds
+d_i^2\int_{0}^{t}\int_{\Omega}\left(\Delta z_i\right)^2dx ds
&=2d_i\int_{0}^{t}\int_{\Omega}\dfrac{\partial z_i}{\partial t}\times \Delta z_idx ds
+\int_{0}^{t}\int_{\Omega}\left(g_i(t,z_i)\right)^2dx ds\\
&=d_i\int_{\Omega}\left(z_i\right)^2dx -d_i\int_{\Omega}\left(z_i^0\right)^2dx.
\end{align*}
Since $g_i\in L^2(Q)$, $z^0_i\in L^2(Q)$ and $z_i,z_i^0\in L_{\infty}(Q)$, 
we obtain that $z_i\in L_{\infty}\left(0;T;H^1(\Omega))\right)$.

Finally, using the same arguments as for the Field--Noyes equations 
in \cite[Example~4]{Smoller},  we deduce that the solution 
$(z_1,z_2,z_3,z_4)$ is nonnegative. Consider the set 
$$
\Sigma=\left\{(z_1,z_2,z_3,z_4): 0\leq z_i\leq C \;\text{ for }\; i\in \{1,2,3,4\}\right\}
$$ 
and the convex functions $G_i$ defined on $\Sigma$ 
by $G_i(z_1,z_2,z_3,z_4)=-z_i$. One can see that
\begin{equation*}
\begin{split}
\nabla(G_1)\cdot g|_{z_1=0}
=\nabla(-z_1)\cdot g|_{z_1=0}
&=-\Lambda-uz_2\leq 0,\\
\nabla(G_2)\cdot g|_{z_2=0}
=\nabla(-z_2)\cdot g|_{z_2=0}
&=-\beta\eta_Cz_3z_1-\beta\eta_Az_4z_1
-\gamma z_4-\omega z_3\leq 0,\\
\nabla(G_3)\cdot g|_{z_3=0}
=\nabla(-z_3)\cdot g|_{z_3=0}
&=-\phi z_1-v_1z_4\leq 0,\\
\nabla(G_4)\cdot g|_{z_4=0}
=\nabla(-z_4)\cdot g|_{z_4=0}
&=-\rho z_2\leq 0.
\end{split}
\end{equation*}
According to \cite[Theorem 14.14]{Smoller}, 
the region $\Sigma$ is positively invariant
and the result follows.
\end{proof}


\section{Existence of an optimal control}
\label{sec:5}

To motivate the interest on optimal control,
we begin by showing some numerical simulations 
of our spatiotemporal SICA model \eqref{SICA2}.
For details on the simulation method, 
tool and used code, see
Appendix~\ref{sec:appendix:code}.

We have considered the values for the parameters 
as given in Table~\ref{tab:2}, 
which were borrowed from \cite{10}.
\begin{table}[ht!]
\caption{Parameters values and units for the SICA model \eqref{SICA2}.}
\label{tab:2}
\begin{center}
\begin{tabular}{lll}
Parameter&Value&Unit\\ \hline
$\mu$&$\dfrac{1}{74.02}$&$day^{-1}$\\
$\Lambda$&$2.19\mu$&$day$\\
$\beta$&$ 0.755$&$(people/km^2)^{-1}.day^{-1}$\\
$\eta_C$&1.5&$day^{-1}$\\
$\eta_A$&0.2&$day^{-1}$\\
$\phi$&1&$day^{-1}$\\
$\rho$&0.1&$day^{-1}$\\
$\gamma$&0.33&$day^{-1}$\\
$\omega$&0.09&$day^{-1}$\\
$d_S$&0.9&$km^2/day$\\
$d_I$&0.1&$km^2/day$\\
$d_C$&0.1&$km^2/day$\\
$d_A$&0.1&$km^2/day$\\
$\xi_1$& $\gamma+\mu$&$day^{-1}$\\
$\xi_2$& $\omega+\mu$&$day^{-1}$\\
$\xi_3$& $\rho+\phi+\mu$&$day^{-1}$\\
\hline
\end{tabular}
\end{center}
\end{table}
Then, the dynamics without control, that is, with $u \equiv 0$ 
in \eqref{SICA2}, is given in Figure~\ref{Fig01}.
\begin{center}
\begin{figure}[ht!]	
\begin{tabular}{cc}
\includegraphics[scale=0.5]{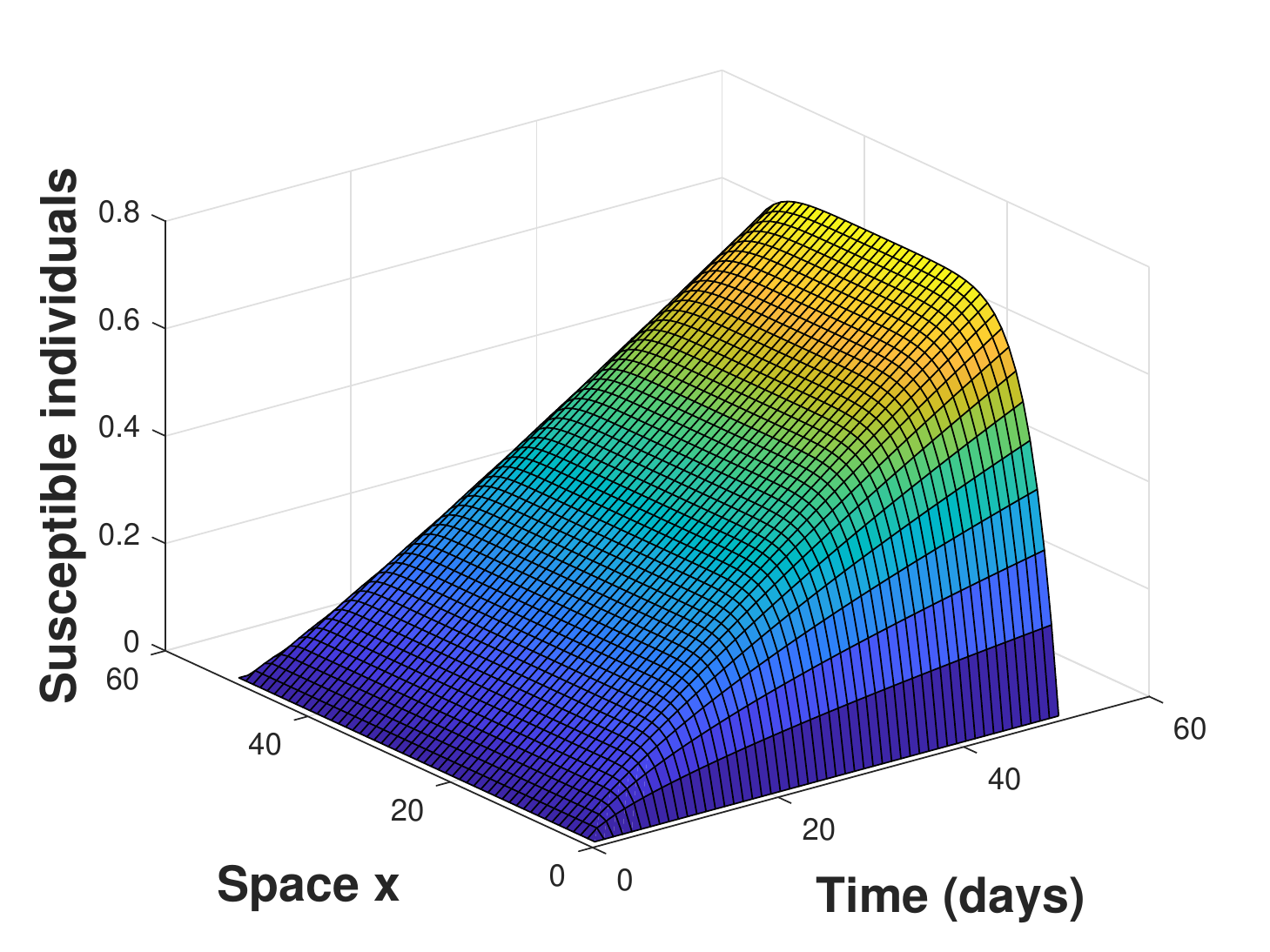}&\includegraphics[scale=0.5]{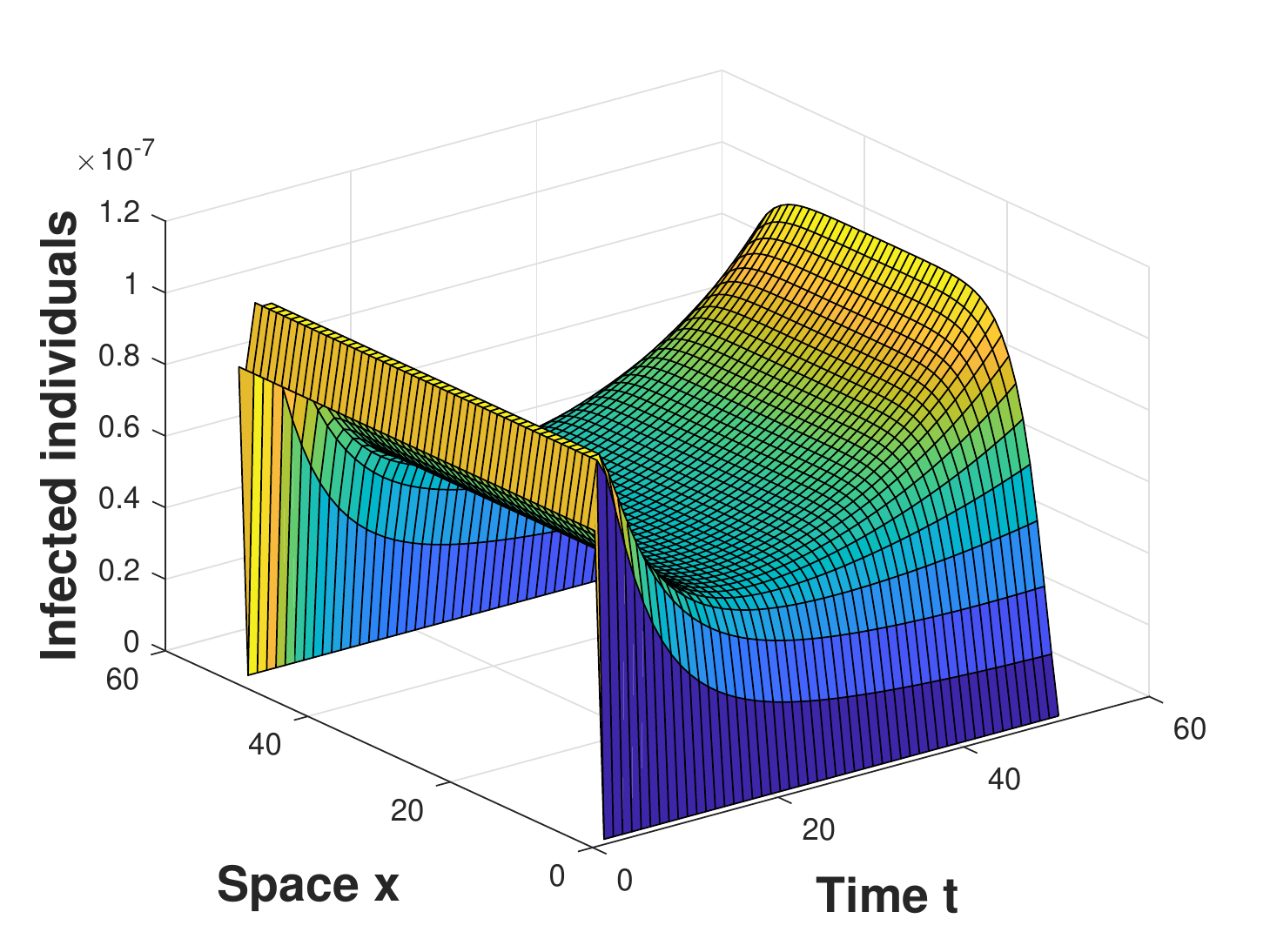}\\
\includegraphics[scale=0.5]{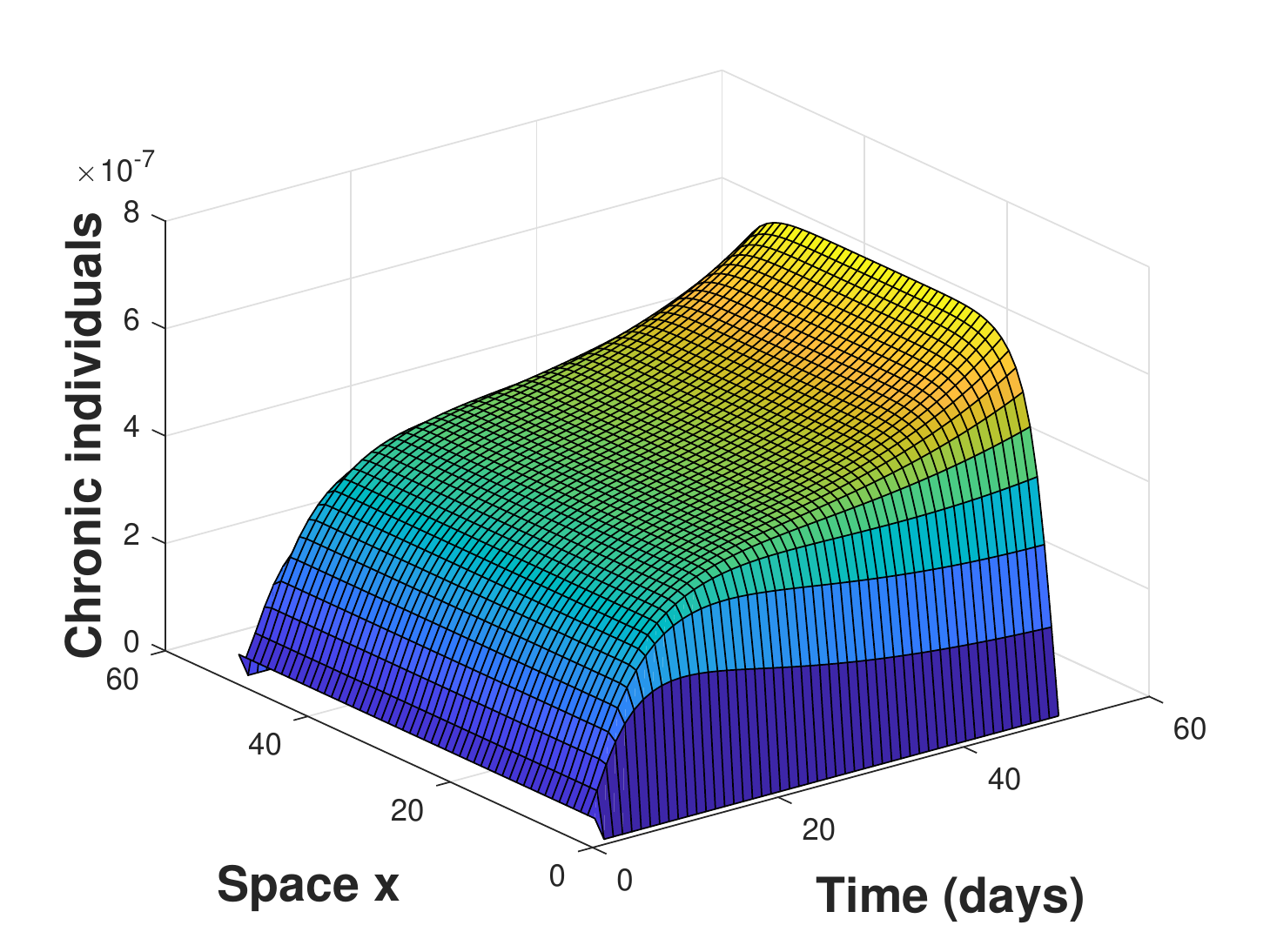}&\includegraphics[scale=0.5]{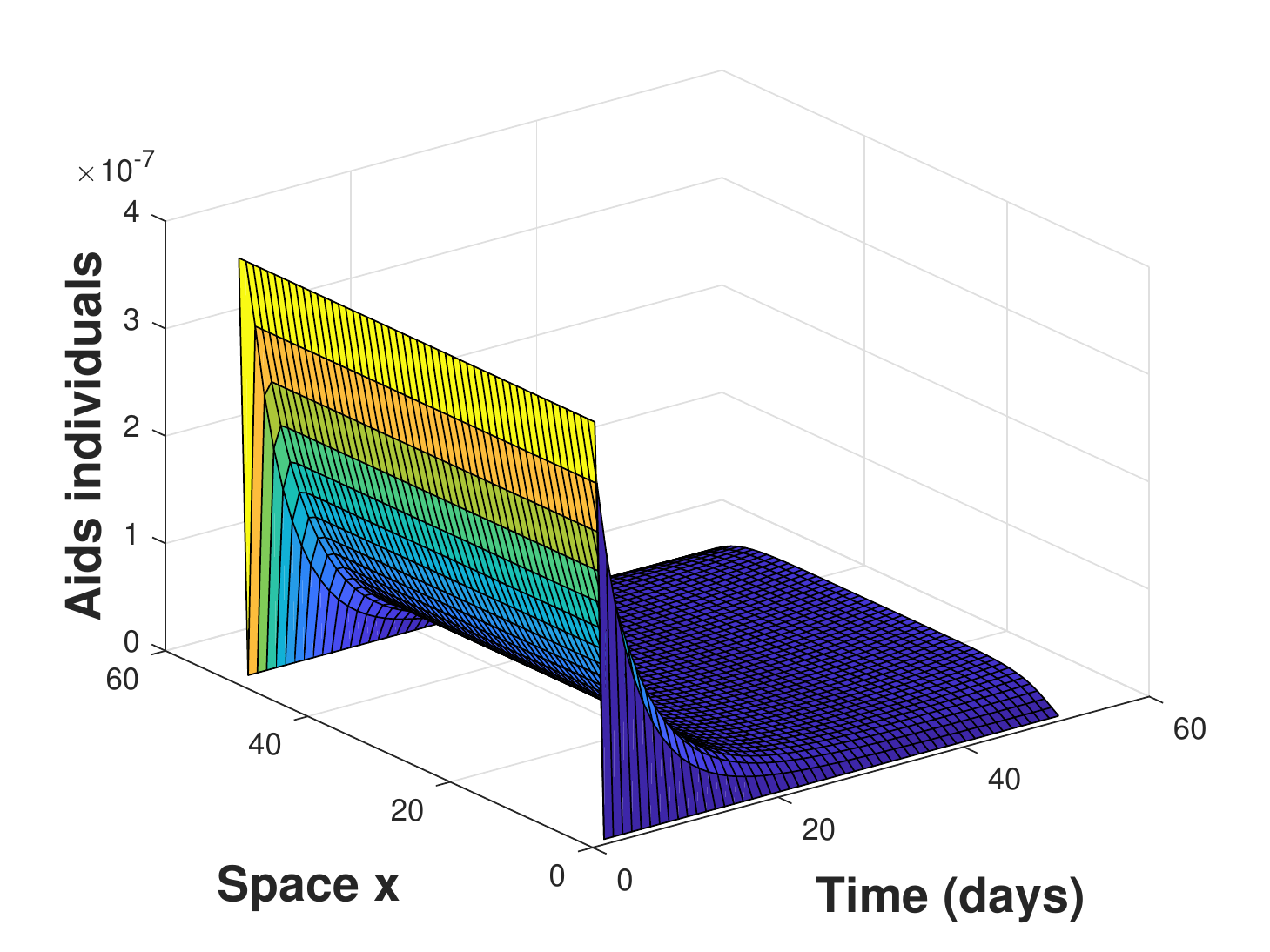}
\end{tabular}
\caption{The behavior of the solution of the system \eqref{SICA2} without control.}
\label{Fig01}
\end{figure}
\end{center}
In contrast, dynamics in the presence of a control are given 
in Figures~\ref{Fig02} and \ref{Fig03}.
\begin{center}
\begin{figure}[ht!]	
\begin{tabular}{cc}
\includegraphics[scale=0.5]{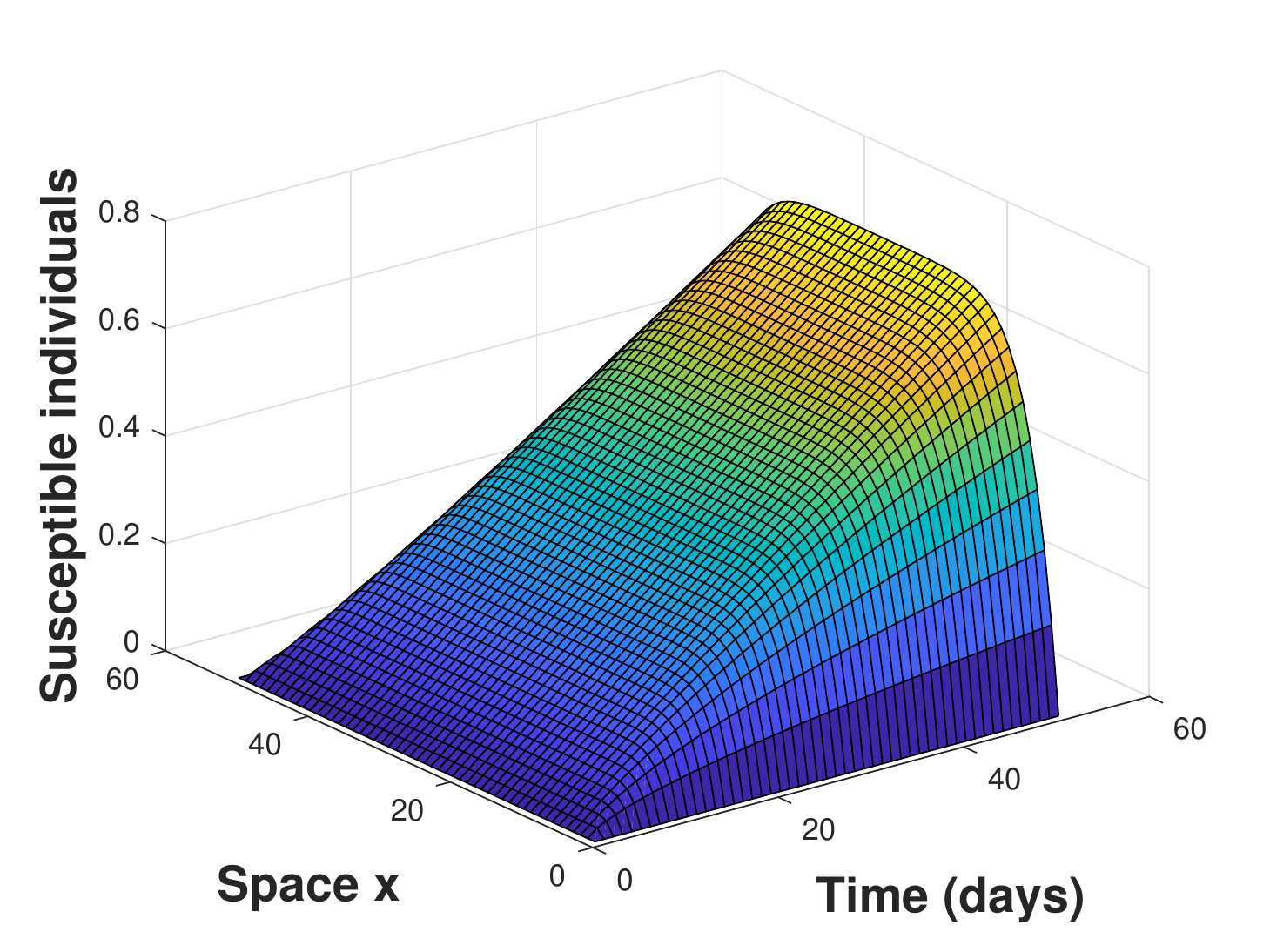}&\includegraphics[scale=0.5]{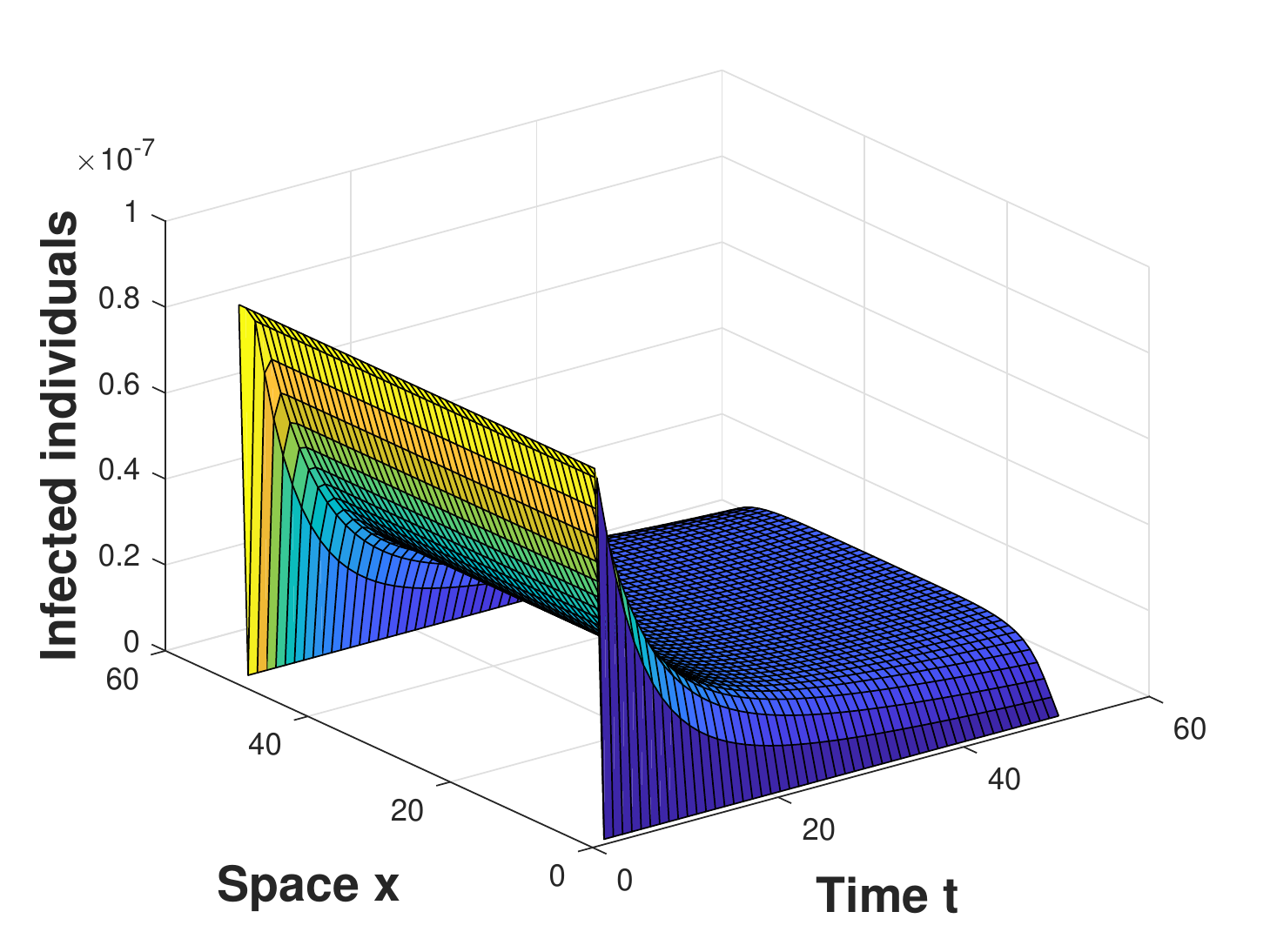}\\
\includegraphics[scale=0.5]{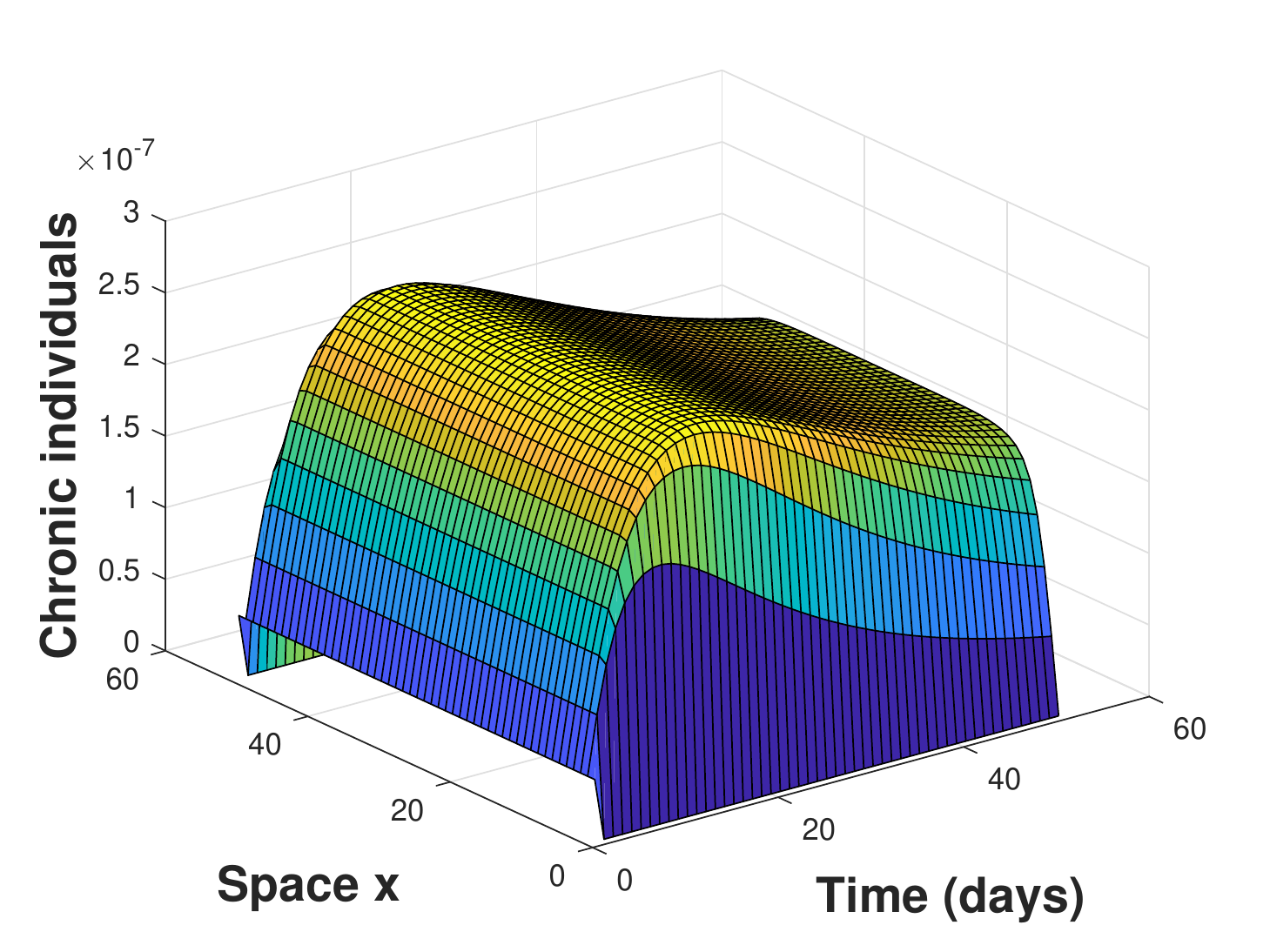}&\includegraphics[scale=0.5]{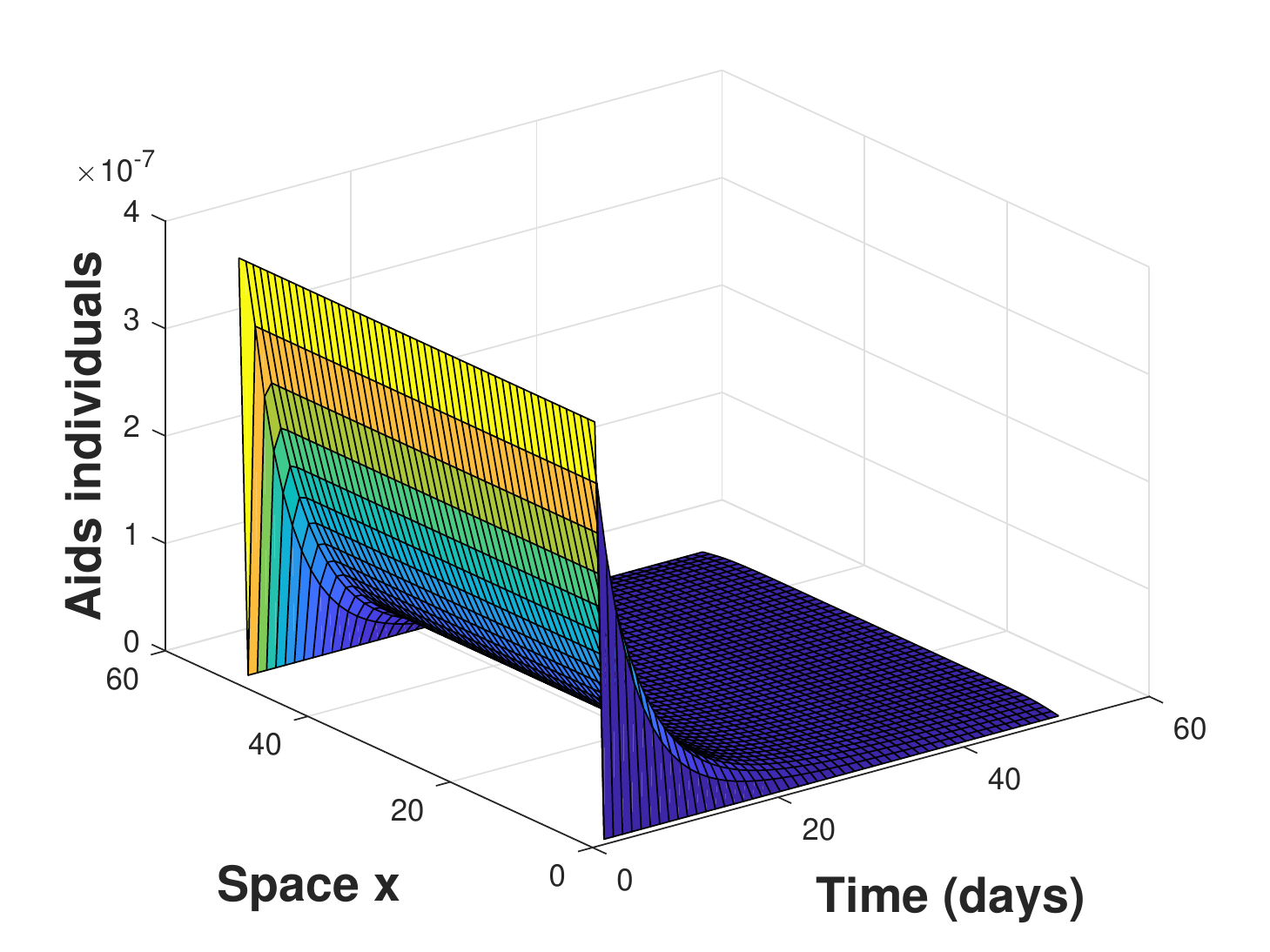}    
\end{tabular}
\caption{The behavior of the solution of the system \eqref{SICA2} with the control $u \equiv 0.5$.}
\label{Fig02}
\end{figure}
\begin{figure}[ht!]
\begin{tabular}{cc}
\includegraphics[scale=0.5]{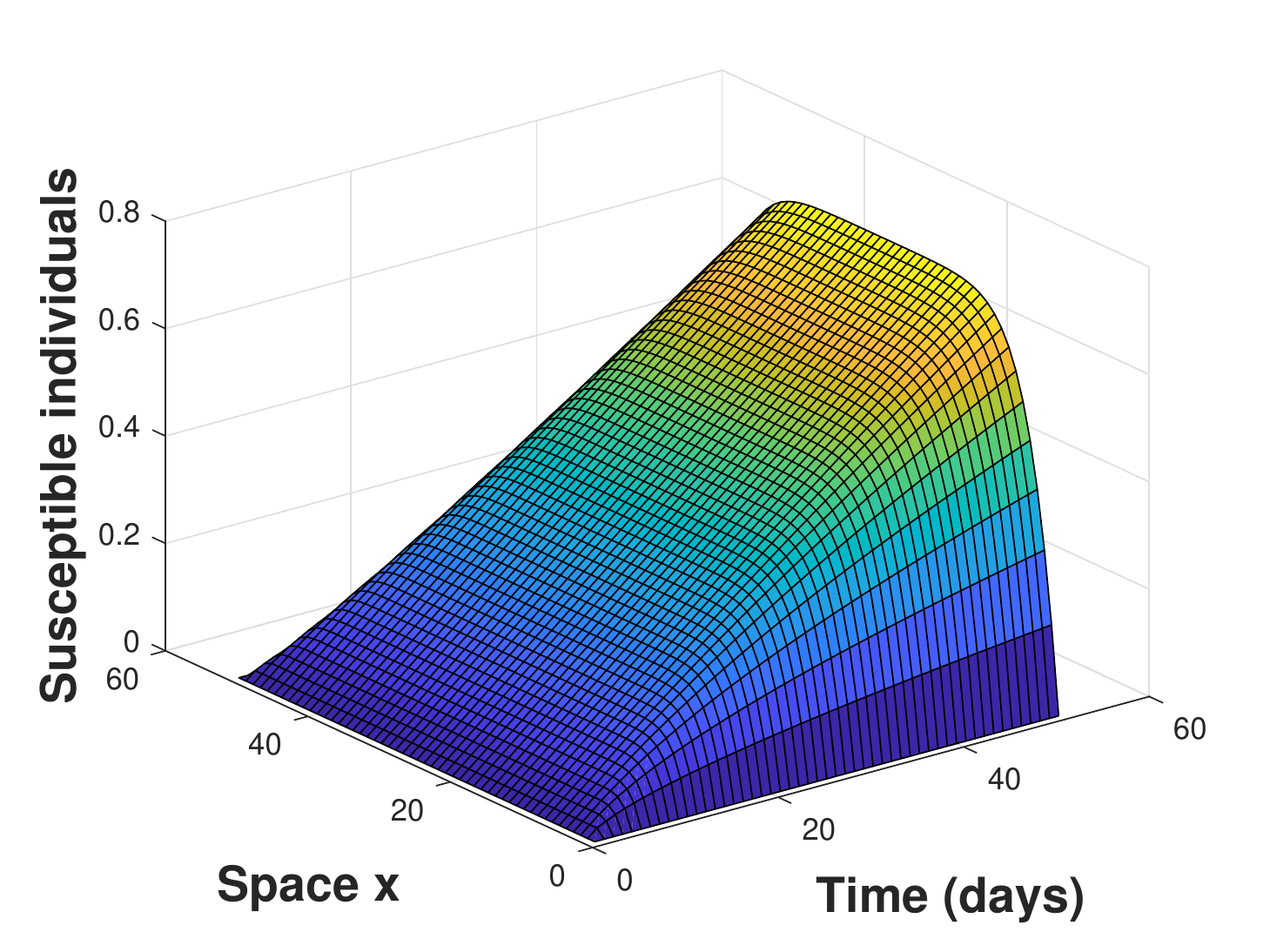}&\includegraphics[scale=0.5]{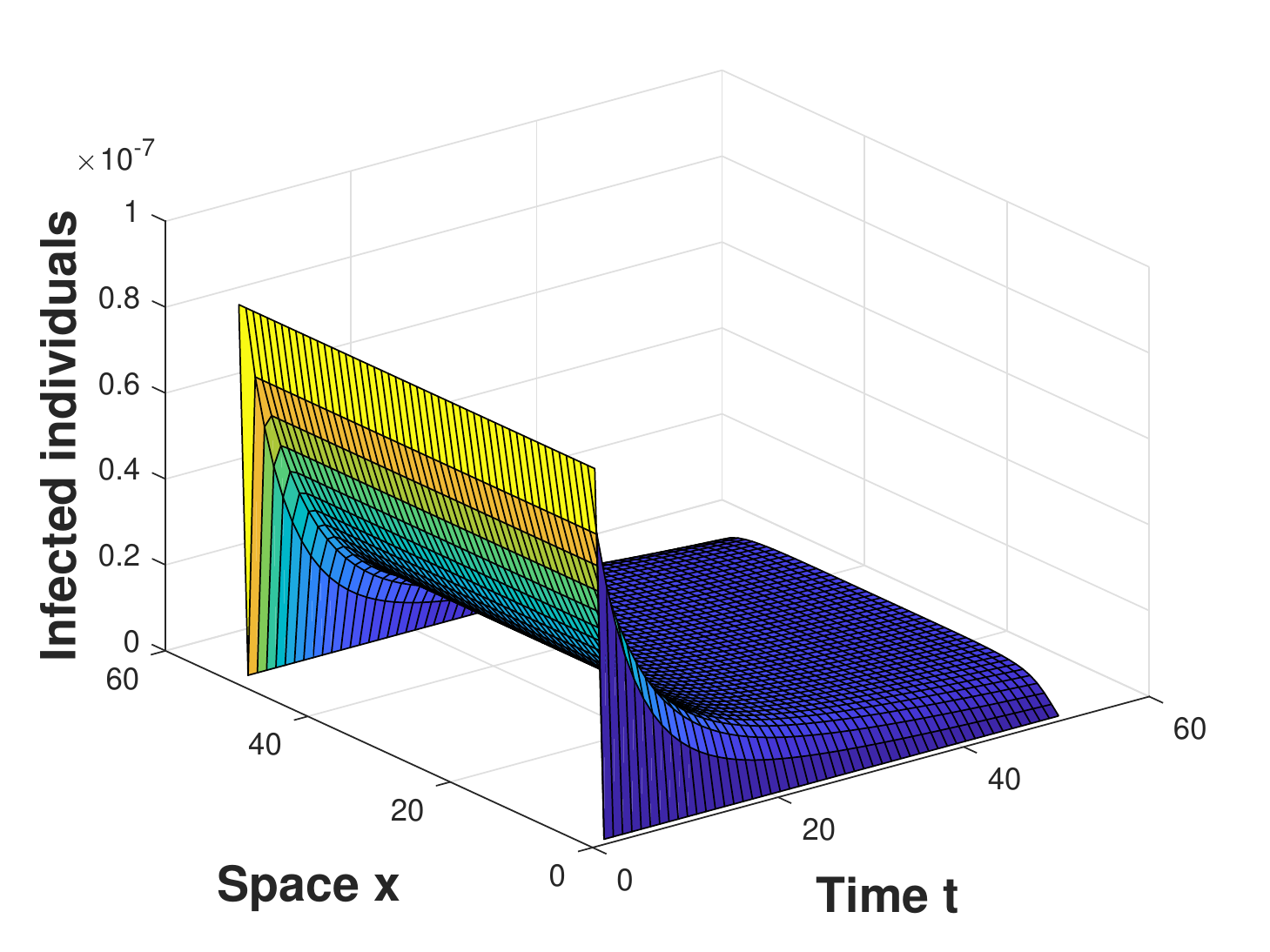}\\
\includegraphics[scale=0.5]{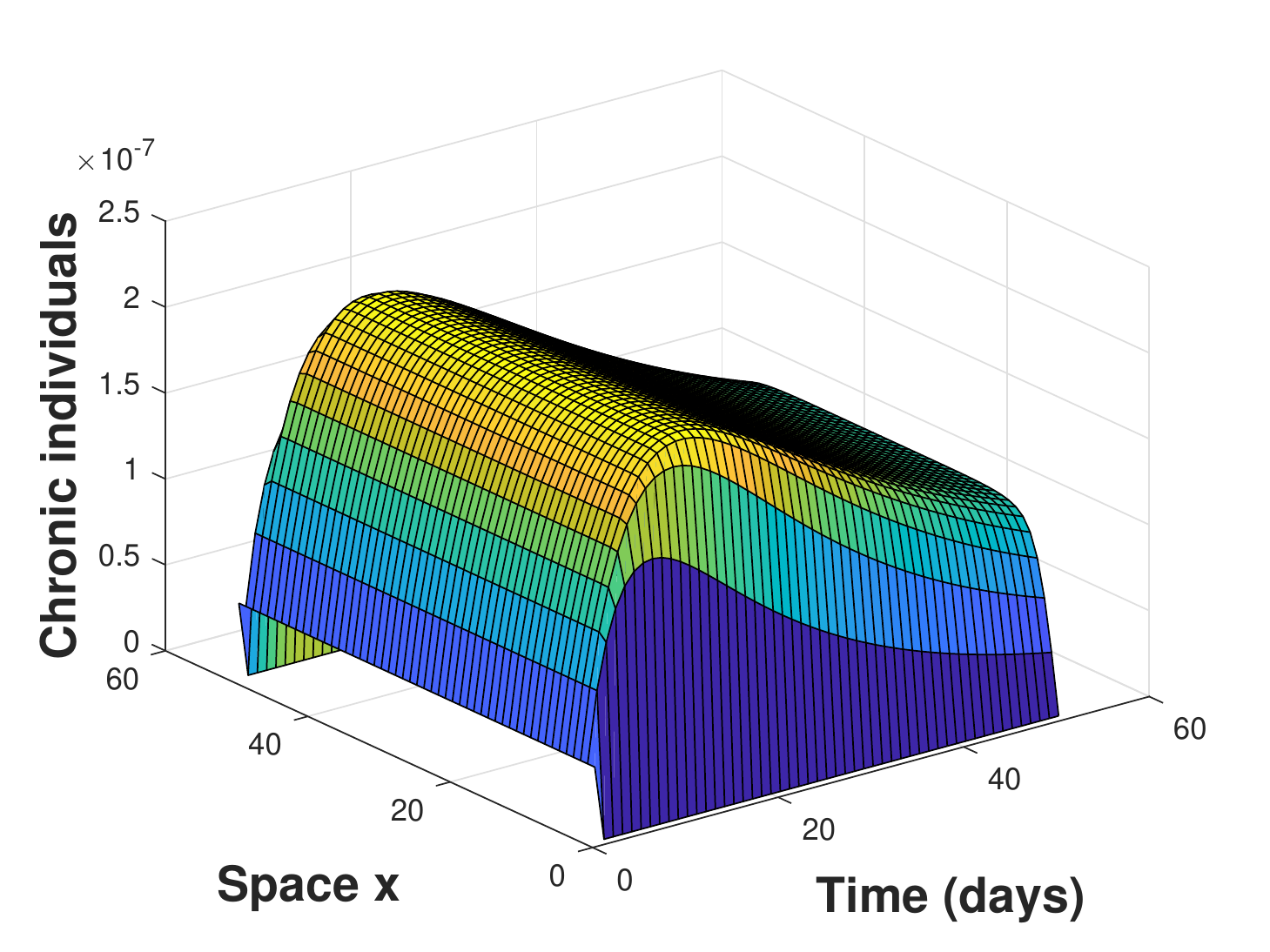}&\includegraphics[scale=0.5]{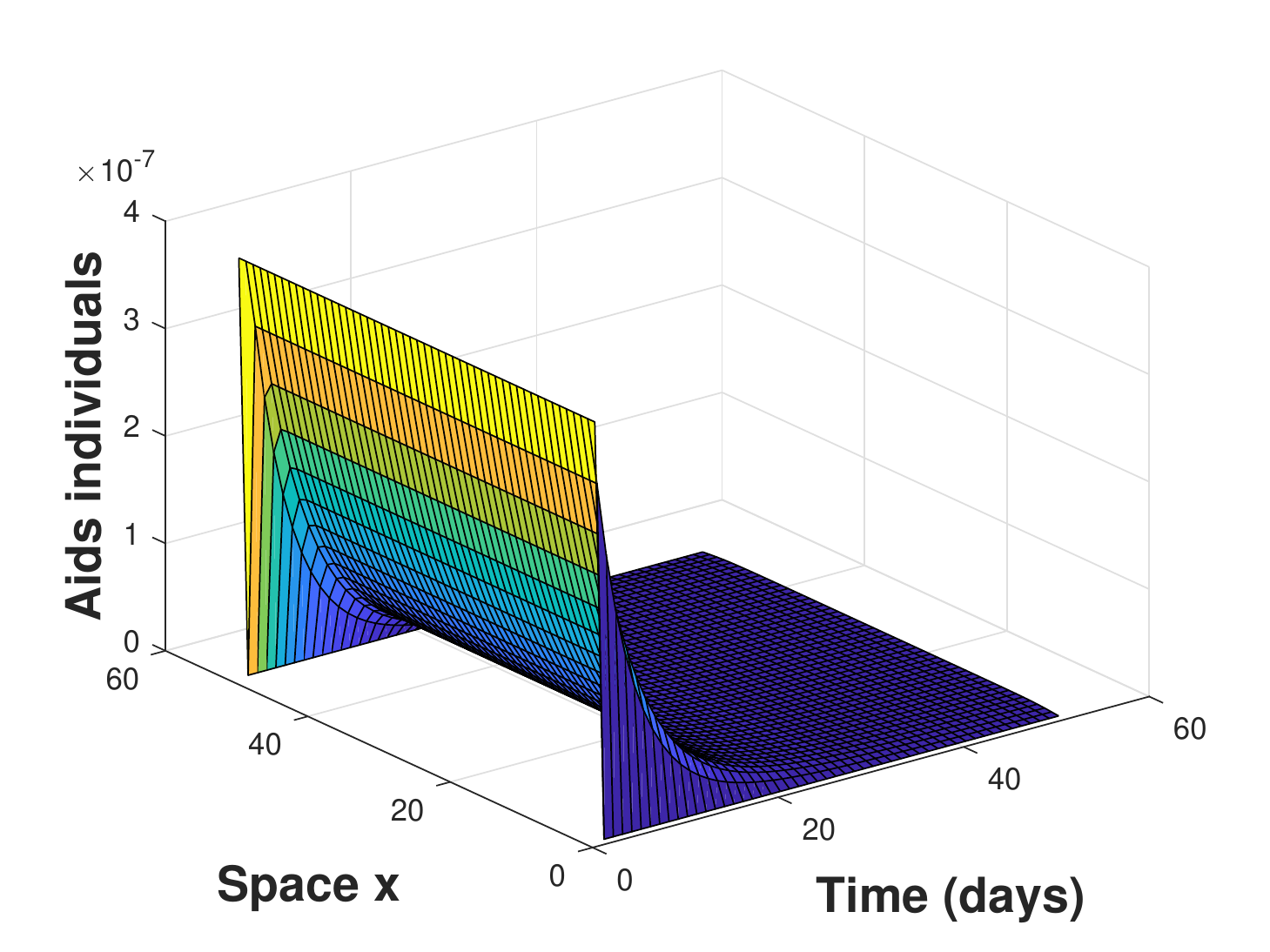}
\end{tabular}
\caption{The behavior of the solution of the system \eqref{SICA2} with the control $u \equiv 0.8$.}
\label{Fig03}
\end{figure}
\end{center}
We conclude that the evolution of the system
related with the absence of control differs 
totally to those in presence of controls. Indeed,
Figure~\ref{Fig01} shows that in absence of the control
the density of the infected individuals increases while in the presence of a control
(Figures~\ref{Fig02} and \ref{Fig03}) it clearly decreases. 
The question of how to choose the control along time, in an optimal way,
is therefore a natural one. 

Motivated by \cite{11}, our aim is to minimize 
the sum of the density of infected individuals and the cost 
of the treatment program. Mathematically,
the problem we consider here is to minimize the objective functional
\begin{equation}
\label{J:functional}
J(S,I,C,A,u)=\int_{\Omega}\int_0^TaI  (t,x)dtdx
+\frac{b}{2} \mid\mid u (t,x)\mid\mid^2_{L^2([0,T])}
\end{equation}
subject to the control system \eqref{SICA2} and where
the admissible control set $U_{ad}$ is defined as in \eqref{Uad}.

\begin{theorem}
\label{thm:02}
Under the conditions of Theorem~\ref{thm}, our optimal
control problem admits a solution $(z^*,u^*)$.
\end{theorem}

\begin{proof}
The proof is divided into three steps.\\

\noindent Step~1: Existence of a minimizing sequence $(z^n,u_n)$.
The infimum of the objective function on the set of admissible 
controls is ensured by the positivity of $J$.
Assume that $J^*=\inf_{u\in U_{ad}}J(z,u)$.
Let $\{u_n\}\subset U_{ad}$  be a minimizing sequence such that 
$\lim\limits_{n\rightarrow +\infty}J(z^n,u_n)=J^*$,
where $(z^n_1,z^n_2,z^n_3,z^n_4)$ is the solution of the system corresponding
to the control $u_n$. Subsequently,
\begin{equation}
\label{systn}
\begin{cases}
\frac{\partial z_1^n}{\partial t} 
= d_S\Delta z_1^n+ \Lambda -\beta \left( z_2^n +\eta _{C}\cdot z_3^n 
+\eta _{A}\cdot z_4^n \right) z_1^n+u (t,x) \cdot z_2^n -\mu  z_1^n,\\ 
\frac{\partial z_2^n}{\partial t} 
=d_I\Delta z_2^n+ \beta \left( z_2^n +\eta _{C}z
\cdot_3^n +\eta _{A}\cdot z_4^n \right) z_1^n -\xi _{3} z_2^n 
+\gamma  z_4^n +\omega  z_3^n-u (t,x)\cdot z_2^n,\\ 
\frac{\partial z_3^n}{\partial t} 
=d_C\Delta  z_3^n + \phi  z_2^n -\xi _{2} z_3^n,\\ 
\frac{\partial z_4^n}{\partial t} 
=d_A\Delta  z_4^n + \rho z_2^n -\xi  _{1} z_4^n, 
\end{cases}
\end{equation}
where $\dfrac{\partial z^n_1}{\partial \eta}=\dfrac{\partial z^n_2}{\partial \eta}
=\dfrac{\partial z^n_3}{\partial \eta}=\dfrac{\partial z^n_4}{\partial \eta}=0$ on $Q$.\\

\noindent Step~2: Convergence of the minimizing sequence $(z^n,u_n)$ to $(z^*, u^*)$.
Let $i\in\{1,2,3,4\}$. Note that $z_i^n (t,x)$ is compact in $L^2(\Omega)$ 
from the fact that $H^1(\Omega)$ is compactly embedded in $L^2(\Omega)$. 
In order to apply the Ascoli--Arzela theorem, we need to demonstrate that 
$\{z_i^n (t,x),n\geq 1\}$ is equicontinuous in $C([0,T],L^2(\Omega))$.
This is indeed true: because of the boundedness of 
$\dfrac{\partial z_i^n}{\partial t}$ in $L^2(Q)$, 
there exists a positive constant $k$ such that 
$$
\left|\int_\Omega (z_i^n)^2  (t,x)dx
- \int_\Omega (z_i^n)^2(s,x)dx \right| 
\leq k \mid t-s \mid 
$$
for all $s,t\in [0,T]$. Hence, $z_i^n$ is compact in $C([0,T],L^2(\Omega))$
and there exists a subsequence of $\{z_i^n\}$, denoted also $\{z_i^n\}$, 
converging uniformly to $z^*_i$ in $L^2(\Omega)$ with respect to $t$.
Since $\Delta z^n_i$ is bounded in $L^2(Q)$, there exists a sub-sequence, 
denoted again $\Delta z^n_i$, converging weakly in $L^2(Q)$. 
For every distribution $\varphi$,
\begin{align*}
\int_Q\varphi \Delta z^n_i 
= \int_Qz^n_i \Delta \varphi \rightarrow \int_Qz^*_i \Delta \varphi
= \int_Q\varphi \Delta z^*_i.
\end{align*}
Thus, $\Delta z^n_i \rightharpoonup \Delta z^*_i$ in $L^2(Q)$.
By the same argument, $\dfrac{\partial z_i^n}{\partial t} 
\rightharpoonup \frac{\partial z^*_i}{\partial t}$ and 
$z_i^n \rightharpoonup z_i^*$ in $L^2(0,T;H^2(\Omega))$ and 
$ z_i^n \rightharpoonup z_i^*$ in $L^\infty(0,T;H^1(\Omega))$.
From $z_1^nz_2^n=(z_1^n-z_1^*)z_2^n+z_1^n(z_2^n-z_2^*)$, we deduce 
that $z_1^nz_2^n \rightharpoonup z_1^*z_2^*$ in $L^2(Q)$. Therefore,  
$u_n \rightharpoonup u^*$ in $L^2(Q)$.
Since $U_{ad}$ is closed, then $u^* \in U_{ad}$.\\

\noindent Step~3: We conclude that $u^n z_2^n \rightharpoonup u^*z_2^*$ in $L^2(Q)$.
Letting $n\rightarrow \infty$ in (\ref{systn}), we obtain that $z^*$ 
is a solution of equation (\ref{eqz}) corresponding to $u^*$. Therefore,
\begin{align*}
J(z^*,u^*)=& \int_0^Taz^*_2 (t,x)dtdx +\frac{b}{2} \mid\mid u^* (t,x)\mid\mid^2_{L^2(Q])}\\
\leq &\liminf \int_0^Taz^n_2 (t,x)dtdx +\frac{b}{2} \mid\mid u^n (t,x)\mid\mid^2_{L^2(Q)}\\
\leq& \lim \int_0^Taz^n_2 (t,x)dtdx +\frac{b}{2} \mid\mid u^n (t,x)\mid\mid^2_{L^2(Q)} = J^*.
\end{align*}
This shows that $J$ attains its minimum at $(z^*,u^*)$.
\end{proof}


\section{Necessary optimality conditions}
\label{sec:6}

Now we characterize the optimality that we proved to exist in Section~\ref{sec:5}. 
Let $(z^*,u^*)$ be an optimal pair and $u^\epsilon = u^*+\epsilon u$, $\epsilon>0$, 
be a control function such that $u \in L^2(Q)$ and $u \in U_{ad}$. We denote by 
$z^\epsilon=(z_1^\epsilon,z_2^\epsilon,z_3^\epsilon,z_4^\epsilon)$ 
and $z^*=(z_1^*,z_2^*,z_3^*,z_4^*)$ the corresponding trajectories
associated with the controls $u^\epsilon $ and $u^*$, respectively.

In the following result we decompose the right-hand side of our 
control system into three quantities: $M$, related to the Laplacian part; 
$R$, linked to the control part; and $F$ for the remaining terms.

\begin{theorem}
\label{thm:3}
For all $i\in\{1,2,3,4\}$, the mapping $u \longmapsto z_i(u)$ defined from $U_{ad}$ 
to $W^{1,2}([0,T],H(\Omega))$ is Gateaux differentiable with respect to $u^*$.
For all $u \in U_{ad}$,  set $z_i'(u^*)u=Z_i$. 
Then $Z=(Z_1,Z_2,Z_3,Z_4)$ is the unique solution of the problem
$$
\dfrac{\partial Z}{\partial t}= M Z+FZ+u R
\quad \text{subject to } \ Z(0,x)=0,
$$
where 
$$
F=\left(
\begin{matrix}
-\beta \left( z_2^* +\eta_{C}\cdot z_3^* 
+\eta _{A}\cdot z_4^* \right) -\mu  &0&0&0\\
\beta \left( z_2^* +\eta_{C}\cdot z_3^* 
+\eta _{A}\cdot z_4^* \right)& -\xi_{3} & \omega &\gamma\\
0&\phi & -\xi _{2}&0\\
0&\rho&0&-\xi  _{1}  
\end{matrix}\right) 
\quad \text{and} \quad
R=\left(
\begin{matrix}
-z_2^*\\
z_2^*\\
0\\
0
\end{matrix}\right).
$$
\end{theorem}

\begin{proof}
Put $Z^\varepsilon_i=\frac{z^\varepsilon_i-z^*_i}{\varepsilon}$. 
By subtracting the two systems verified by 
$z^\varepsilon_i$ and $z^*_i$, we get
$$ 
\frac{\partial Z^\varepsilon}{\partial t}
= M Z^\varepsilon+FZ^\varepsilon+u R \;\; \text{subject to}\;\; 
Z^\varepsilon(0,x)=0,\;\;\text{for all}\; x\in \Omega.
$$
Consider the semigroup $(\Gamma(t), t \geq 0)$ generated by $M$.
Then the solution of this system is given by
$$
Z^\varepsilon (t,x)
=\int_0^t\Gamma(t-s)FZ^\varepsilon(s,x)ds+\int_0^t \Gamma(t-s)u R ds.
$$
Since the elements of the matrix $F^\varepsilon$ are uniformly bounded 
with respect to $\varepsilon$, according to Gr\"{o}nwall's inequality one has that
$Z^\varepsilon_i$ is bounded in $L^2(Q)$. Hence, $z^\varepsilon_i \rightarrow z^*_i$
in $L^2(Q)$. Letting $ \varepsilon \rightarrow 0 $, we have 
$$
\dfrac{\partial Z}{\partial t}= M Z+FZ+u R
\quad \text{ subject to } \quad Z(0,x)=0,\; 
\text{ for all } x\in \Omega.
$$
Adopting the same technique, we deduce that 
$Z^\varepsilon_i \rightarrow Z^*_i$ as $\varepsilon \rightarrow 0$.
\end{proof}

Let $p = (p_1, p_2, p_3,p_4)$ be the adjoint variable of $Z$ and denote by
$F^*$ the adjoint of the Jacobian matrix $F$. We can write the
dual system associated to our problem as
\begin{equation}
\label{eq:dual}
-\dfrac{\partial p}{\partial t}- M p-F^*p=D^*D\psi 
\quad \text{ subject to } p(T,x)=0,
\end{equation}
where
$$
D=
\begin{pmatrix}
0&0&0&0\\
0&1&0&0\\
0&0&0&0\\
0&0&0&0
\end{pmatrix}
\quad \text{and} \quad 
\psi=
\begin{pmatrix}
0\\
a\\
0\\
0
\end{pmatrix}. 
$$

\begin{lemma}
Under the hypothesis of Theorem~\ref{thm}, the system \eqref{eq:dual}
of adjoint variables admits a unique solution $p \in W^{1,2}([0, T],H(\Omega))$ 
with $ p_i \in  G(T,\Omega)$, $i = 1,2,3,4$.
\end{lemma}

\begin{proof}
The result follows by the change of variables $s=T-t$ so as to apply 
the same method performed in the proof of Theorem~\ref{thm:3}.
\end{proof}

We are now in a position to obtain a necessary optimality condition 
for the optimal control $u^*$.

\begin{theorem}
\label{thm:5}
If $u^*$ is an optimal control and $z^*\in W^{1,2}([0,T];H(\Omega))$  
is its  corresponding solution, then
\begin{equation}
\label{eq:u:min:max}
u^*=\min \left( u_{\max},\max \left( 0,\dfrac{z_2^*(p_2-p_1}{b}\right) \right). 
\end{equation}
\end{theorem}

\begin{proof}
Let $u^*$ be an optimal control and let $z^*$ 
be the corresponding optimal state. 
Set $u^\varepsilon=u^*+\varepsilon u \in U_{ad}$ 
and let $z^\varepsilon$ be the corresponding state trajectory. We have
\begin{align*}
J'(u^*)(u)
&=\lim\limits_{\varepsilon\rightarrow 0}
\dfrac{1}{\varepsilon}\left(J(u^{\varepsilon})-J(u^*)\right)\\
&=\lim\limits_{\varepsilon\rightarrow 0}
\dfrac{1}{\varepsilon}\left(a\int_0^T\int_\Omega
\left(z_2^\varepsilon-z_2^*\right)dxdt+\dfrac{b}{2}
\int_0^1\int_\Omega\left((u^\varepsilon)^2-(u^*)^2\right)dxdt\right)\\
&=\lim\limits_{\varepsilon\rightarrow 0}\left(a\int_0^T\int_\Omega
\left(\dfrac{z_2^\varepsilon-z_2^*}{\varepsilon}\right)dxdt
+\dfrac{b}{2}\int_0^1\int_\Omega\left(2 u u^*+\varepsilon u^2\right)dxdt\right).
\end{align*}
Since $\lim\limits_{\varepsilon\rightarrow 0}
\dfrac{z_2^\varepsilon-z_2^*}{\varepsilon}
=\lim\limits_{\varepsilon\rightarrow 0}
\dfrac{z_2(u^*+\varepsilon h)-z_2^*}{\varepsilon}=Z_2$, 
$\lim\limits_{\varepsilon\rightarrow 0} z_2^\varepsilon=z_2^*$ 
and $z_2^\varepsilon,z_2^*\in L^\infty(Q)$, then $J$ is Gateaux differentiable 
with respect to $u^*$ with
\begin{align*}
J'(u^*)(u)&=\int_0^T\int_\Omega a Z_2 dxdt+b\int_0^T\int_\Omega u u^*dxdt\\
&=\int_0^T\langle D\psi ,DZ \rangle dt
+\int_0^1\langle bu^* ,u \rangle_{L^2(\Omega) }dt.
\end{align*} 
If we take $u=v-u^*$, then we obtain
$$
J'(u^*)(v-u^*)=\int_0^T\langle D\psi ,DZ \rangle dt
+\int_0^1\langle bu^* ,v-u^* \rangle_{L^2(\Omega) }dt.
$$
Since
\begin{align*}
\int_0^T\langle D\psi ,DZ \rangle dt
&=\int_0^T\left\langle D^*D\psi ,Z \right\rangle dt\\
&=\int_0^T\left\langle -\dfrac{\partial p}{\partial t}-M p-F^*p ,Z \right\rangle dt\\
&=\int_0^T\left\langle p ,\dfrac{\partial Z}{\partial t }-MZ-FZ \right\rangle dt\\
&=\int_0^T\left\langle p ,R(v-u^*) \right\rangle dt\\
&=\int_0^T\left\langle R^*p ,v-u^* \right\rangle_{L^2(\Omega) }dt
\end{align*}
and $U_{ad}$ is convex, then $J'(u^*)(v-u^*)\geq 0$ 
for all $v\in U_{ad}$, which is equivalent to
$$
\int_0^T\langle R^*p+bu^* ,v-u^* \rangle_{L^2(\Omega) }dt
\geq 0 \text{ for all } v\in U_{ad}.
$$
Thus, $bu^*=-{R^*} p$ and, consequently, $u^*=\dfrac{z_2^*(p_2-p_1)}{b}$. 
Since $u^*\in U_{ad}$, we have that \eqref{eq:u:min:max} holds.
\end{proof}

Note that Theorem~\ref{thm:5}
provides a constructive method, giving an explicit 
expression \eqref{eq:u:min:max} for the optimal control.


\section{Conclusion and future work}
\label{sec:7}

We have extended the time deterministic epidemic SICA model due 
to Silva and Torres \cite{10} to spatiotemporal dynamics, 
which take into account not only the local reaction of 
appearance of new infected individuals but also the global diffusion
occurrence of the other infected individuals. This allows
to incorporate an additional amount of arguments into the system. 
More precisely, firstly we have modeled the spatiotemporal behavior by 
incorporating the well-known Laplace operator, which has been 
employed in the literature, in different contexts, to better 
understand what happens during any possible displacement 
of different species and individuals. Here, we justify and interpret
its use in the context of HIV/AIDS epidemics. Secondly, we have presented 
an optimal control problem to minimize the number of infected individuals 
through a suitable cost functional. Proved results include: existence 
and uniqueness of a strong global solution to the system, 
obtained using some adapted tools from semigroup theory;
some characteristics of the existing solution;
existence of an optimal control, investigated
using an effective method based on some properties within the weak topology;
and necessary optimality conditions to quantify explicitly the optimal control.

As future work, we plan to develop numerical methods
for spatiotemporal optimal control problems, implementing
the necessary optimality conditions we have proved here.
This is under investigation and will be addressed elsewhere.
Another interesting line of research concerns the bifurcation 
analysis for different parameters.


\appendix

\section{Simulation method and code}
\label{sec:appendix:code}

The focus of our work is more theoretical, 
linked to the proposed spatiotemporal SICA epidemic model \eqref{SICA2}.
In Section~\ref{sec:5}, to motivate our study on optimal control, 
we have incorporated some selected control values in order 
to present some adequate scenarios showing the dynamic evolution 
of the system. In our simulations, we have adopted the first 
order explicit Euler method to discretize the temporal derivatives 
and the second order explicit Euler method to discretize 
the Laplacian operator. Follows our \textsf{Octave}/\textsf{Matlab} code:

\small 
\begin{lstlisting}
clc; clear;
dS=0.1; dI=0.1; dC=0.1; dA=0.1;
beta=0.755;
mu=1/74.02;
lamda=(2.19)*mu;
phi=1;
rho=0.1;
omega=0.09;
gamma=0.33;
eta1=1.5; eta2=0.2;
x1=gamma+mu;
x2=omega+mu;
x3=rho+phi+mu;
u=0.5;
N=50;
h=20;
tf=10;
xmax=1;
xmin=0;
deltax=(xmax-xmin)/N+1;
deltat=(tf-0)/h;
S=zeros(N+1,N+1);
S(1,1)=100/23023935;
I=zeros(N+1,N+1);
I(1,1)=2/23023935;
A=zeros(N+1,N+1);
A(1,1)=9/23023935;
C=zeros(N+1,N+1);
C(1,1)=0;

for j=1:N
 for i=2:N
  S(i,j+1)=S(i,j)+((deltat*dS)/(deltax*deltax))*(S(i+1,j)-2*S(i,j)+S(i-1,j)) ...
    +lamda*deltat-(beta*deltat)*(I(i,j)+eta1*A(i,j)+eta2*C(i,j))*S(i,j) ...
    -(mu*deltat)*S(i,j)+u*deltat*I(i,j);
  I(i,j+1)=I(i,j)+(deltat*dI/(deltax*deltax))*(I(i+1,j)-2*I(i,j)+I(i-1,j)) ...
    +beta*deltat*(I(i,j)+eta1*A(i,j)+eta2*C(i,j))*S(i,j)-x3*deltat*I(i,j) ...
    +gamma*deltat*A(i,j)+omega*deltat*C(i,j)-u*deltat*I(i,j);
  C(i,j+1)=C(i,j)+(deltat*dC/(deltax*deltax))*(C(i+1,j)-2*C(i,j)+C(i-1,j)) ...
    +phi*deltat*I(i,j)-x2*deltat*C(i,j);
  A(i,j+1)=A(i,j)+(deltat*dA/(deltax*deltax))*(A(i+1,j)-2*A(i,j)+A(i-1,j)) ...
    +rho*deltat*I(i,j)-x1*deltat*A(i,j);
  end
end

figure
surf(S);
xlabel('Time (days)','FontSize',6,'FontWeight','bold');
ylabel('Space x','FontSize',6,'FontWeight','bold');
zlabel('Susceptible individuals','FontSize',6,'FontWeight','bold');

figure
surf(I);
xlabel('Time t','FontSize',6,'FontWeight','bold');
ylabel('Space x','FontSize',6,'FontWeight','bold');
zlabel('Infected individuals','FontSize',6,'FontWeight','bold');

figure
surf(C);
xlabel('Time t','FontSize',6,'FontWeight','bold');
ylabel('Space x','FontSize',6,'FontWeight','bold');
zlabel('Clinical individuals','FontSize',6,'FontWeight','bold');

figure
surf(A);
xlabel('Time t','FontSize',6,'FontWeight','bold');
ylabel('Space x','FontSize',6,'FontWeight','bold');
zlabel('Aids individuals','FontSize',6,'FontWeight','bold');
\end{lstlisting}

The reader interested in the scientific computing tool
\textsf{GNU Octave} or \textsf{Matlab} is referred to \cite{MR3495870}.


\section*{Acknowledgments}

This research was funded by The Portuguese Foundation for Science and Technology 
(FCT---Funda\c{c}\~{a}o para a Ci\^{e}ncia e a Tecnologia), grant number UIDB/04106/2020 (CIDMA).
The authors are very grateful to three anonymous Reviewers for several constructive questions
and remarks that helped them to improve their work.


\section*{Conflict of interest}

The authors declare that there are no conflicts of interest.



\end{document}